\documentclass[12pt,reqno,a4paper]{amsart}

\usepackage{amsfonts, amsmath, amsthm, amssymb}
\usepackage[english]{babel}
\usepackage[mathcal]{eucal}
\usepackage{graphicx}

\allowdisplaybreaks

\title[]{Approximate central limit theorems}
\author{Ben Berckmoes and Geert Molenberghs}
\thanks{Ben Berckmoes is post doctoral fellow at the Fund for Scientific Research of Flanders (FWO)}
\thanks{Geert Molenberghs gratefully acknowledges financial support from the IAP research network  \#P7/06 of the Belgian Government (Belgian Science Policy)}

\date{}

\DeclareMathOperator*{\myinf}{in\vphantom{p}f}

\begin{document}

\maketitle

\newtheorem{pro}{Proposition}[section]
\newtheorem{lem}[pro]{Lemma}
\newtheorem{thm}[pro]{Theorem}
\newtheorem{de}[pro]{Definition}
\newtheorem{co}[pro]{Comment}
\newtheorem{no}[pro]{Notation}
\newtheorem{vb}[pro]{Example}
\newtheorem{vbn}[pro]{Examples}
\newtheorem{gev}[pro]{Corollary}
\newtheorem{vrg}[pro]{Question}
\newtheorem{rem}[pro]{Remark}
\newtheorem{lemA}{Lemma}

\begin{abstract}
We refine the classical Lindeberg-Feller central limit theorem by obtaining asymptotic bounds on the Kolmogorov distance, the Wasserstein distance, and the parametrized Prokhorov  distances in terms of a Lindeberg index. We thus obtain more general approximate central limit theorems, which roughly state that the row-wise sums of a triangular array are approximately asymptotically normal if the array approximately satisfies Lindeberg's condition. This allows us to continue to provide information in non-standard settings in which the classical central limit theorem fails to hold. Stein's method plays a key role in the development of this theory.
\end{abstract}

\section{Introduction}\label{sec:Intro}
Throughout, we assume that all random variables are defined on a fixed probability space $(\Omega,\mathcal{F},\mathbb{P})$.

Let $\xi$ be a standard normal random variable, that is, a normally distributed random variable with $\mathbb{E}[\xi] = 0$ and $\mathbb{E}[\xi^2] = 1$, and $\{\xi_{n,k}\}$ a standard triangular array (STA) of random variables, that is, a triangular array
\begin{equation*}
\begin{array}{cccc} 
\xi_{1,1} &  &  \\
\xi_{2,1} & \xi_{2,2} & \\
\xi_{3,1} & \xi_{3,2} & \xi_{3,3} \\
 & \vdots &
\end{array}
 \end{equation*}
of random variables with $\xi_{n,1}, \ldots, \xi_{n,n}$ independent for all $n$, $\mathbb{E}[\xi_{n,k}] = 0$ for all $n,k$, and $\sum_{k=1}^n \mathbb{E}[\xi_{n,k}^2] = 1$ for all $n$.
 
 Recall that the sequence $\left(\sum_{k=1}^n \xi_{n,k}\right)_n$ is said to converge weakly to $\xi$ iff
 \begin{equation*}
 \lim_{n \rightarrow \infty} \mathbb{P}\left[\sum_{k=1}^n \xi_{n,k} \leq x_0\right] = \mathbb{P}[\xi \leq x_0]
 \end{equation*}
 for all $x_0$ at which the map $x \mapsto \mathbb{P}[\xi \leq x]$ is continuous, or, equivalently, iff
 \begin{equation*}
 \lim_{n \rightarrow \infty} \mathbb{E}\left[h\left(\sum_{k=1}^n \xi_{n,k}\right)\right] = \mathbb{E}[h(\xi)]
 \end{equation*}
 for all $h : \mathbb{R} \rightarrow \mathbb{R}$ bounded and continuous.
 
 We say that $\{\xi_{n,k}\}$ satisfies Feller's condition iff
 \begin{equation}
 \lim_{n \rightarrow \infty} \max_{k=1}^n \mathbb{E}\left[\xi_{n,k}^2\right] = 0,\label{eq:FelCon}
 \end{equation}
 and Lindeberg's condition iff 
 \begin{equation*}
\lim_{n \rightarrow \infty} \sum_{k=1}^n \mathbb{E}\left[\xi_{n,k}^2 ; \left|\xi_{n,k}\right| > \epsilon\right] = 0
 \end{equation*}
 for all $\epsilon > 0$. It is easily seen that Lindeberg's condition implies Feller's, but that the converse does not hold.
 
The above language allows us to formulate the following result, which belongs to the heart of classical probability theory.
 
 \begin{thm}[Lindeberg-Feller Central Limit Theorem]\label{thm:LFCLT}
 Let $\xi$ and $\{\xi_{n,k}\}$ be as above. If $\{\xi_{n,k}\}$ satisfies Lindeberg's condition, then the sequence $(\sum_{k=1}^n \xi_{n,k})_n$ converges weakly to $\xi$. The converse holds if $\{\xi_{n,k}\}$ satisfies Feller's condition.
 \end{thm}
 
 The number
\begin{equation}
\textrm{\upshape{Lin}}\left(\{\xi_{n,k}\}\right) = \sup_{\epsilon > 0} \limsup_{n \rightarrow \infty} \sum_{k=1}^n \mathbb{E}\left[\xi_{n,k}^2 ; \left|\xi_{n,k}\right| > \epsilon\right]\label{eq:Lin}
\end{equation}
was introduced in \cite{BLV13} as the Lindeberg index. Notice that it produces for each STA a number between 0 and 1, and that it is 0 if and only if Lindeberg's condition is satisfied. So it can be thought of as a number which measures how far a given STA deviates from satisfying Lindeberg's condition.

Furthermore, let $d(\eta,\eta^\prime)$ be a metric on random variables with the property that $ \lim_{n \rightarrow \infty} d(\eta,\eta_n) = 0$ is equivalent with weak convergence of $(\eta_n)_n$ to $\eta$, and define the quantity
\begin{equation}
\lambda_{d}\left(\sum_{k=1}^n \xi_{n,k} \rightarrow \xi\right) = \limsup_{n \rightarrow \infty} d\left(\xi,\sum_{k=1}^n \xi_{n,k}\right).\label{eq:LimOpd}
\end{equation}
Clearly, (\ref{eq:LimOpd}) assigns a positive number to each STA which is 0 if and only if the row-wise sums of the STA are asymptotically normal. Thus this number measures how far a given STA deviates from having an asymptotically normal sequence of row-wise sums.

Now, using the numbers (\ref{eq:Lin}) and (\ref{eq:LimOpd}), the first part of Theorem \ref{thm:LFCLT} leads to the implication
\begin{equation*}
\textrm{\upshape{Lin}}(\{\xi_{n,k}\}) = 0 \Rightarrow \lambda_{d}\left(\sum_{k=1}^n \xi_{n,k} \rightarrow \xi\right) = 0.
\end{equation*}
Observe that Theorem \ref{thm:LFCLT} fails to provide any information for the large class of STA's which fail to satisfy Lindeberg's condition, regardless of whether $\textrm{\upshape{Lin}}(\{\xi_{n,k}\})$ is large or small. Thus the following natural question arises.

\begin{vrg}\label{vrg:vrg}
Suppose that we are given an STA $\{\xi_{n,k}\}$ which is close to satisfying Lindeberg's condition in the sense that $\textrm{\upshape{Lin}}(\{\xi_{n,k}\})$ is non-zero but small. Is it still possible to conclude that the row-wise sums of $\{\xi_{n,k}\}$ are close to being asymptotically normal in the sense that $\lambda_{d}\left(\sum_{k=1}^n \xi_{n,k} \rightarrow \xi\right) $ is small?
\end{vrg}

Let us briefly describe how in the case where $d$ is the Kolmogorov metric
 \begin{equation*}
 K\left(\eta,\eta^\prime\right) = \sup_{x \in \mathbb{R}} \left|\mathbb{P}[\eta \leq x] - \mathbb{P}\left[\eta^\prime \leq x\right]\right|,
 \end{equation*}
 a positive answer to Question \ref{vrg:vrg} can be derived from the existing literature.

The following refinement of the sufficiency of Lindeberg's condition in Theorem \ref{thm:LFCLT} was obtained in terms of the Kolmogorov distance in \cite{O66} and \cite{F68}.

\begin{thm}\label{thm:OsiFeller}
Let $\xi$ be as above. Then there exists a universal constant $C > 0$ such that 
\begin{equation*}
K\left(\xi,\sum_{k=1}^n \xi_{n,k}\right) \leq C\left( \sum_{k=1}^n \mathbb{E}\left[\xi_{n,k}^2 ; \left|\xi_{n,k}\right| > 1\right]  + \sum_{k=1}^n\mathbb{E}\left[\left|\xi_{n,k}\right|^3 ; \left|\xi_{n,k}\right| \leq 1\right]\right)
\end{equation*}
for all STA's $\{\xi_{n,k}\}$ and all $n$.
\end{thm} 

It was shown in \cite{F68} that the constant $C$ in Theorem \ref{thm:OsiFeller} can be taken equal to 6. A proof of Theorem \ref{thm:OsiFeller} based on Stein's method was given in \cite{BH84}, and in \cite{CS01}, combining Stein's method with Chen's concentration inequality approach, it was established that $C$ can be taken equal to 4.1, the best value known so far up to our knowledge.

We will infer a corollary from Theorem \ref{thm:OsiFeller} which is related to Question \ref{vrg:vrg}. To this end, we remark that it was pointed out in \cite{L75} that the truncation at 1 in Theorem \ref{thm:OsiFeller} is optimal in the sense that
\begin{equation*}
\sum_{k=1}^n \mathbb{E}\left[\xi_{n,k}^2 ; \left|\xi_{n,k}\right| > 1\right]  + \sum_{k=1}^n\mathbb{E}\left[\left|\xi_{n,k}\right|^3 ; \left|\xi_{n,k}\right| \leq 1\right]
\end{equation*}
is dominated by 
\begin{equation*}
\sum_{k=1}^n \mathbb{E}\left[\xi_{n,k}^2 ; \xi_{n,k} \in A\right]  + \sum_{k=1}^n\mathbb{E}\left[\left|\xi_{n,k}\right|^3 ; \xi_{n,k} \in \mathbb{R} \setminus A\right]
\end{equation*}
for each Borel set $A \subset \mathbb{R}$. Therefore, we easily derive from Theorem \ref{thm:OsiFeller} that
\begin{equation*}
K\left(\xi,\sum_{k=1}^n \xi_{n,k}\right) \leq C \left( \sum_{k=1}^n \mathbb{E}\left[\xi_{n,k}^2 ; \left|\xi_{n,k}\right| > \epsilon\right]  + \epsilon\right)
\end{equation*}
for all $\epsilon > 0$, which, calculating the superior limit of both sides and letting $\epsilon \downarrow 0$, yields
\begin{equation*}
\limsup_{n \rightarrow \infty} K\left(\xi,\sum_{k=1}^n \xi_{n,k}\right) \leq C \sup_{\epsilon > 0} \limsup_{n \rightarrow \infty} \sum_{k=1}^n \mathbb{E}\left[\xi_{n,k}^2 ; \left|\xi_{n,k}\right| > \epsilon\right].
\end{equation*}

Using the numbers defined in (\ref{eq:Lin}) and (\ref{eq:LimOpd}), we now derive the following result as a corollary of Theorem \ref{thm:OsiFeller}.

\begin{thm}\label{thm:ACLTK}
Let $\xi$ be as above. Then there exists a universal constant $C > 0$ such that 
\begin{equation*}
\lambda_K\left(\sum_{k=1}^n \xi_{n,k} \rightarrow \xi\right) \leq C \textrm{\upshape{Lin}}\left(\{\xi_{n,k}\}\right)
\end{equation*}
for all STA's $\{\xi_{n,k}\}$.
\end{thm}

\begin{rem}\label{rem:ACLTK}
In \cite{BLV13}, combining Stein's method with an asymptotic smoothing technique, it was established that the constant $C$ in Theorem \ref{thm:ACLTK} can be taken equal to 1 if $\{\xi_{n,k}\}$ satisfies Feller's condition.
\end{rem}

Notice that Theorem \ref{thm:ACLTK} gives a positive answer to Question \ref{vrg:vrg} in the case where $d = K$. It strictly generalizes the sufficiency of Lindeberg's condition in Theorem \ref{thm:LFCLT}, and, contrary to Theorem \ref{thm:LFCLT}, it continues to provide useful information for STA's which have a low Lindeberg index, but fail to satisfy Lindeberg's condition. More precisely, it allows us to conclude that $\left(\sum_{k=1}^n \xi_{n,k}\right)_n$ is approximately convergent to $\xi$ if $\{\xi_{n,k}\}$ approximately satisfies Lindeberg's condition. Therefore, it seems plausible to refer to Theorem \ref{thm:ACLTK} as an approximate central limit theorem.

The problem of generalizing Theorem \ref{thm:OsiFeller} to the multivariate setting is hard, and remains open. Notice however that recently, combining a multivariate version of Stein's method, as outlined in e.g. \cite{M09} and \cite{NPR10}, with the establishment of an explicit integral representation of a solution to the Stein PDE with a character function as test function, a partial extension of Theorem \ref{thm:ACLTK} for the Fourier transforms of random vectors has been obtained in \cite{BLV}.

In this paper, we will focus on the following question concerning Theorem \ref{thm:ACLTK}.\\

{\em Can we widen the scope of applicability of Theorem \ref{thm:ACLTK} by extending it to other probability metrics $d$?}

The paper is structured as follows.

A short overview of some important probability metrics is given in section 2.

In section 3, we show that it is possible to apply the techniques used in \cite{BLV13} to a large class of test functions, leading to a general inequality. 

The inequality presented in section 3 leads to approximate central limit theorems, similar to Theorem \ref{thm:ACLTK}, for the Wasserstein distance and the parametrized Prokhorov distances. These are given in section 4. An example shows that a result of the same flavor cannot be obtained for the total variation distance.

\section{Some probability metrics}

Let $\mathcal{P}(\mathbb{R})$ be the collection of Borel probability measures on $\mathbb{R}$.  Furthermore, let $\mathcal{P}_{1}(\mathbb{R})$ be the set of all  $P \in \mathcal{P}(\mathbb{R})$ with finite absolute first moment, i.e. for which $\int_{-\infty}^\infty \left|x\right| dP(x) < \infty$.

The Wasserstein distance on $\mathcal{P}_1(\mathbb{R})$, see e.g. \cite{V03}, is defined by the formula
\begin{equation*}
W(P,Q) = \myinf_{\pi} \int_{\mathbb{R} \times \mathbb{R}} d(x,y) d\pi(x,y),
\end{equation*} 
where the infimum is taken over all Borel probability measures $\pi$ on $\mathbb{R} \times \mathbb{R}$ with first marginal $P$ and second marginal $Q$. Kantorovich duality theory implies that the metric $W$ can also be written as
\begin{equation}
W(P,Q) = \sup_{h \in \mathcal{K}(\mathbb{R})} \left|\int_\mathbb{R} h dP - \int_\mathbb{R} h dQ\right|\label{rep:WassersteinDual},
\end{equation}
where $\mathcal{K}(\mathbb{R})$ stands for the set of all contractions $h : \mathbb{R} \rightarrow \mathbb{R}$, where $h$ is called a contraction iff $\left|h(x) - h(y)\right| \leq \left|x - y\right|$ for all $x,y \in \mathbb{R}$. Also, we have
\begin{equation*}
W(P,Q) = \int_{-\infty}^\infty \left|F_P(x) - F_Q(x)\right| dx = \int_0^1 \left|F_P^{-1}(t) - F_Q^{-1}(t)\right|dt,
\end{equation*}
with $F_P$ (respectively $F_Q$) the cumulative distribution function associated with $P$ (respectively $Q$), and $F_P^{-1}$ (respectively $F_Q^{-1}$) its generalized inverse.

The topology underlying the Wasserstein distance is slightly stronger than the weak topology. More precisely, for $P$ and $(P_n)_n$ in $\mathcal{P}_1(\mathbb{R})$, it holds that $\lim_{n \rightarrow \infty}W(P,P_n) = 0$ is equivalent with weak convergence of $(P_n)_n$ to $P$ in addition to convergence of $\left(\int_{-\infty}^\infty \left|x\right| dP_n(x)\right)_n$ to $\int_{-\infty}^\infty \left|x\right| dP(x)$. Also, the Wasserstein distance is separable and complete, see \cite{B08}.

Furthermore, for $\lambda \in \mathbb{R}^+_0$, the (parametrized) Prokhorov distance $\rho_\lambda(P,Q)$\index{$\rho_\lambda(P,Q)$} between probability measures $P$ and $Q$ in $\mathcal{P}(\mathbb{R})$ is defined to be the infimum of all positive numbers $\alpha \in  \mathbb{R}^+_0$ for which the inequality
\begin{equation*}
P\left[A\right] \leq Q\left[A^{(\lambda \alpha)}\right] + \alpha,\label{def:ProkhMetric}
\end{equation*}
with 
\begin{equation*}
A^{(\lambda \alpha)} = \left\{x \in \mathbb{R} \mid \inf_{a \in A} \left|x - a\right| \leq \lambda \alpha\right\},\index{$A^{(\alpha)}$}
\end{equation*}
holds for every Borel set $A \subset \mathbb{R}$. One easily establishes that 
\begin{equation*}
\rho_{\lambda_1}(P,Q) \leq \rho_{\lambda_2}(P,Q)\label{ProkhorovIncreases}
\end{equation*}
whenever $\lambda_2 \leq \lambda_1$. In \cite{B99} it is shown that, for each $\lambda \in \mathbb{R}^+_0$, $\rho_\lambda$ is a separable and complete metric which metrizes weak convergence of probability measures.

Finally, the total variation distance $d_{TV}(P,Q)$\index{$d_{TV}(P,Q)$} between probability measures $P$ and $Q$ in $\mathcal{P}(\mathbb{R})$ is defined by the number
\begin{equation*}
d_{TV}(P,Q) = \sup_{A} \left|P[A] - Q[A]\right|,
\end{equation*} 
the supremum of course taken over all Borel sets $A \subset \mathbb{R}$. One easily verifies that $d_{TV}$ is a complete metric, that, for each $\lambda \in \mathbb{R}^+_0$,
\begin{equation*}
\rho_\lambda(P,Q) \leq d_{TV}(P,Q),
\end{equation*}
and that the limit relation
\begin{equation}
\lim_{\lambda \downarrow 0} \rho_\lambda(P,Q) = d_{TV}(P,Q)\label{LimProkhTV}
\end{equation}
holds true. Note however that $d_{TV}$ is not separable and that its underlying topology is strictly stronger than the weak topology.

For a general and systematic treatment of the theory of probability metrics, we refer the reader to the excellent expositions \cite{Z83} and \cite{R91}.

\section{A general inequality}

Let $\xi$ be as in Section \ref{sec:Intro} and $h : \mathbb{R} \rightarrow \mathbb{R}$ a continuous map for which $\mathbb{E}\left|h(\xi)\right| < \infty$. Then the Stein transform of $h$ is the map $f_h : \mathbb{R} \rightarrow \mathbb{R}$ defined by the formula
\begin{equation}
f_h(x) = e^{x^2/2} \int_{-\infty}^x \left(h(t) - \mathbb{E}[h(\xi)]\right) e^{-t^2/2} dt.\label{eq:SteinTransform}
\end{equation}
The crux of Stein's method is that, for any random variable $\eta$, we have
\begin{equation*}
\mathbb{E}\left[h(\xi) - h(\eta)\right] = \mathbb{E}[\eta f_h(\eta) - f_h^\prime(\eta)],
\end{equation*}
and that, in many cases, it is easier to find upper bounds for the derivatives of $f_h$ than for the derivatives of $h$, see e.g. \cite{BC05} and \cite{CGS11}.

We will now establish a general inequality in terms of the Stein transform, which will allow us to extend Theorem \ref{thm:ACLTK} to many of the above described probability metrics. For the proof, it basically suffices to notice that the techniques developed in \cite{BLV13} can be extended to a very general collection of test functions. For the sake of completeness, we present the proof in Appendix A.

\begin{thm}\label{thm:BigIneq}
Let $\xi$ and $\{\xi_{n,k}\}$ be as in Section \ref{sec:Intro}, and let $h : \mathbb{R} \rightarrow \mathbb{R}$ be any continuously differentiable map with a bounded derivative. Then the Stein transform $f_h$, defined by (\ref{eq:SteinTransform}), is twice continuously differentiable, has bounded first and second derivatives, and the inequality
\begin{eqnarray}
\lefteqn{\left|\mathbb{E}\left[h(\xi) - h\left(\sum_{k=1}^n \xi_{n,k}\right)\right]\right|}\label{eq:BigIneq}\\
&\leq& \frac{1}{2}\|f_h^{\prime \prime}\|_\infty \epsilon + \left(\sup_{x_1,x_2 \in \mathbb{R}} \left|f_h^\prime(x_1) - f_h^\prime(x_2)\right|\right) \sum_{k=1}^n \mathbb{E}\left[\xi_{n,k}^2 ; \left|\xi_{n,k}\right| \geq \epsilon\right]\nonumber\\
&&+ \left(\sup_{x_1,x_2 \in \mathbb{R}} \left|f_h^{\prime \prime}(x_1) - f_h^{\prime \prime}(x_2)\right|\right) \max_{k=1}^n \mathbb{E}\left[\left|\xi_{n,k}\right|\right]\nonumber
\end{eqnarray}
holds for all $n$ and all $\epsilon > 0$.
\end{thm}

\section{Approximate central limit theorems}

We will apply Theorem \ref{thm:BigIneq} to obtain results similar to Theorem \ref{thm:ACLTK} for the Wasserstein distance (Theorem \ref{thm:ACLTW}) and the parametrized Prokhorov distances (Theorem \ref{thm:ACLTP}). Where needed, we tacitly transport these probability metrics to random variables via their image measures. 

The following lemma guarantees that we can capture the Wasserstein distance with continuously differentiable contractions.

\begin{lem}\label{lem:SmoothingWasserstein}
The Wasserstein distance on $\mathcal{P}_1(\mathbb{R})$ is given by
\begin{equation}
W(P,Q) = \sup_{h \in \mathcal{K}_c(\mathbb{R})} \left|\int h dP - \int h dQ\right|,\label{eq:WassSmooth}
\end{equation}
where $\mathcal{K}_c(\mathbb{R})$ stands for the set of all continuously differentiable contractions $h : \mathbb{R} \rightarrow \mathbb{R}$. 
\end{lem}

\begin{proof}
Let $h : \mathbb{R} \rightarrow \mathbb{R}$ be a contraction and fix $\epsilon > 0$. We will show that there exists a smooth contraction which is closer than $\epsilon$ to $h$ for the $\|\cdot\|_\infty$-norm. Once this is established, the lemma will follow from formula (\ref{rep:WassersteinDual}).

Let $$\psi_{\epsilon} : \mathbb{R} \rightarrow \mathbb{R}$$ be positive and smooth, with support contained in the interval $[-\epsilon,\epsilon]$, and such that $$\int_{\mathbb{R}} \psi_{\epsilon}(y) dy = 1.$$ Put
\begin{displaymath}
h_{\epsilon}(x) = (h \star \psi_{\epsilon})(x) = \int_{\mathbb{R}} h(x - y) \psi_{\epsilon}(y) dy = \int_{\mathbb{R}} \psi_{\epsilon}(x - y) h(y) dy.
\end{displaymath}
Then $h_{\epsilon}$ is smooth. Furthermore, for $x_1,x_2 \in \mathbb{R}$,
\begin{eqnarray*}
\lefteqn{\left|h_\epsilon(x_1) - h_\epsilon(x_2)\right|}\\
&=&  \left|\int_{\mathbb{R}} h(x_1 - y) \psi_{\epsilon}(y) dy - \int_{\mathbb{R}} h(x_2 - y) \psi_{\epsilon}(y) dy\right|\\
&\leq& \int_{\mathbb{R}} \left|h(x_1 - y) - h(x_2 - y)\right| \psi_{\epsilon}(y) dy,
\end{eqnarray*}
which is, $h$ being a contraction, bounded by $\int_{\mathbb{R}} \psi_{\epsilon}(y) dy = 1,$ and we infer that $h_{\epsilon}$ is also a contraction. Finally, for $x \in \mathbb{R}$, 
\begin{eqnarray}
\lefteqn{\left|h(x) - h_\epsilon(x)\right|}\nonumber\\
&=& \left|\int_{\mathbb{R}} \left(h(x) - h(x - y)\right) \psi_{\epsilon}(y) dy\right|\nonumber\\
&=& \left|\int_{-\epsilon}^\epsilon \left(h(x) - h(x-y)\right) \psi_\epsilon(y) dy\right|,\label{eq:hheleq1}
\end{eqnarray}
the last equality following from the fact that the support of $\psi_\epsilon$ is contained in $[-\epsilon,\epsilon]$. Now, $h$ being a contraction, it follows that the expression in (\ref{eq:hheleq1}) is bounded by $\epsilon$, whence 
$$\|h - h_\epsilon\|_\infty < \epsilon.$$
This concludes the proof.
\end{proof}

The following lemma belongs to the basics of Stein's method, see e.g. \cite{BC05}, p.10-11.

\begin{lem}\label{lem:ACLTWbounds}
Let $h$ and $f_h$ be as in Theorem \ref{thm:BigIneq}. Then
\begin{equation}
\|f_h^\prime\|_\infty \leq 4 \|h^\prime\|_\infty\label{eq:lemACLTW1}
\end{equation}
and
\begin{equation}
\|f_h^\prime\|_\infty \leq 2 \|\mathbb{E}[h(\xi)] - h\|_\infty\label{eq:lemACLTP} 
\end{equation}
and
\begin{equation}
\|f_{h}^{\prime \prime}\|_\infty \leq 2 \|h^\prime\|_\infty.\label{eq:lemACLTW2}
\end{equation}
\end{lem}

\begin{thm}\label{thm:ACLTW}
Let $\xi$ be as in Section \ref{sec:Intro}. Then there exists a universal constant $C_W > 0$ such that 
\begin{equation*}
\lambda_W\left(\sum_{k=1}^n \xi_{n,k} \rightarrow \xi\right) \leq C_W \textrm{\upshape{Lin}}\left(\{\xi_{n,k}\}\right)
\end{equation*}
for all STA's $\{\xi_{n,k}\}$ which satisfy Feller's condition (\ref{eq:FelCon}). Moreover, $C_W$ can be taken equal to $8$.
\end{thm}

\begin{proof}
Let $h : \mathbb{R} \rightarrow \mathbb{R}$ be a continuously differentiable contraction. Then
\begin{equation}
\sup_{x_1,x_2 \in \mathbb{R}} \left|f_h^\prime(x_1) - f_h^\prime(x_2)\right| \leq 2 \|f_h^\prime\|_\infty \leq 8 \|h^\prime\|_\infty \leq 8,\label{eq:ACLTW1}
\end{equation}
by (\ref{eq:lemACLTW1}), and
\begin{equation}
\sup_{x_1,x_2 \in \mathbb{R}} \left|f_h^{\prime \prime}(x_1) - f_h^{\prime \prime} (x_2)\right| \leq 2 \|f_h^{\prime \prime}\|_\infty \leq 4 \|h^\prime\|_\infty \leq 4\label{eq:ACLTW2},
\end{equation}
by (\ref{eq:lemACLTW2}). Furthermore, combining (\ref{eq:BigIneq}) with (\ref{eq:lemACLTW2}), (\ref{eq:ACLTW1}), and (\ref{eq:ACLTW2}), yields
\begin{eqnarray}
\lefteqn{\left|\mathbb{E}\left[h(\xi) - h\left(\sum_{k=1}^n \xi_{n,k}\right)\right]\right|}\label{eq:ACLTW3}\\
&\leq&  \epsilon + 8 \sum_{k=1}^n \mathbb{E}\left[\xi_{n,k}^2 ; \left|\xi_{n,k}\right| \geq \epsilon\right] + 4 \max_{k=1}^n \mathbb{E}\left[\left|\xi_{n,k}\right|\right]\nonumber
\end{eqnarray}
for all STA's $\{\xi_{n,k}\}$, all $n$, and all $\epsilon > 0$. Finally, assuming that $\{\xi_{n,k}\}$ satisfies Feller's condition, taking the supremum over all $h \in \mathcal{K}_c(\mathbb{R})$, calculating the superior limits, and letting $\epsilon \downarrow 0$, we see that that (\ref{eq:WassSmooth}) and (\ref{eq:ACLTW3}) lead to the desired result.
\end{proof}

Lemma \ref{lem:Ph} reveals that we can capture all parametrized Prokhorov distances by one collection of smooth test functions. It can be derived indirectly from \cite{BLV11} (Section 2, Lemma 2.2), where the so-called weak approach structure on the set of probability measures on a separable metric space was studied, see also \cite{L15}. As it is a crucial step to obtain an approximate central limit theorem for the parametrized Prokhorov distances, we will present a direct proof here.

\begin{lem}\label{lem:Ph}
Let $\mathcal{H}$ be the collection of continuously differentiable maps $h : \mathbb{R} \rightarrow [0,1]$ with a bounded derivative. Then, for $P$ and $(P_n)_n$ in $\mathcal{P}(\mathbb{R})$,
\begin{equation}
\sup_{\lambda > 0} \limsup_{n \rightarrow \infty} \rho_\lambda(P,P_n) = \sup_{h \in \mathcal{H}} \limsup_{n \rightarrow \infty} \left|\int_{\mathbb{R}} h dP - \int_{\mathbb{R}} h dP_n\right|.\label{eq:Ph}
\end{equation}
\end{lem}

\begin{proof}
First suppose that, for $\gamma > 0$,
\begin{equation}
\sup_{h \in \mathcal{H}} \limsup_{n \rightarrow \infty} \left|\int_{\mathbb{R}} h dP - \int_{\mathbb{R}} h dP_n\right| < \gamma.\label{eq:seven}
\end{equation}
Now fix $\epsilon > 0$ and $\lambda > 0$, and choose real numbers 
$$a_1 < a_2 <  \cdots < a_{N-1} < a_N$$ 
such that
\begin{equation}
\forall k \in \{1,\ldots,N-1\} : a_{k+1} - a_k < \lambda \gamma /2\label{eq:five}
\end{equation}
and
\begin{equation}
P\left[\mathbb{R} \setminus [a_1,a_N]\right] < \epsilon.\label{eq:one}
\end{equation}
For each subset $K \subset \{1,\ldots,N-1\}$, choose $h_K \in \mathcal{H}$ such that 
\begin{equation}
\forall x \in \cup_{k \in K} \left[a_k,a_{k+1}\right] : h_K(x) = 1\label{eq:two}
\end{equation}
and
\begin{equation}
\forall x \in \mathbb{R} \setminus \left(\cup_{k \in K} \left[a_k,a_{k+1}\right]\right)^{(\lambda \gamma /2)} : h_K(x) = 0.\label{eq:four}
\end{equation}
By (\ref{eq:seven}), for each $K \subset \{1,\ldots,N-1\}$, there exists $n_K$ such that for all $n \geq n_K$
\begin{equation}
\left|\int_\mathbb{R} h_K dP - \int_\mathbb{R} h_K dP_n\right| < \gamma.\label{eq:three}
\end{equation}
Let $n_0 = \max_{K \subset \{1,\ldots,N-1\}} n_K$, and take $n \geq n_0$ and a Borel set $A \subset \mathbb{R}$. Put
\begin{equation}
K_0 = \left\{ k \in \{1,\ldots,N-1\} : A \cap \left[a_k,a_{k+1}\right] \neq \emptyset\right\}.\label{eq:six}
\end{equation}
Then
$$P[A] \leq P\left[\cup_{k \in K_0} ]a_k,a_{k+1}[ \right] + P\left[\mathbb{R} \setminus [a_1,a_N]\right],$$
which, by (\ref{eq:one}) and (\ref{eq:two}),
$$\leq \int_\mathbb{R} h_{K_0} dP + \epsilon,$$
which, by (\ref{eq:three}),
$$< \int_\mathbb{R} h_{K_0} dP_n + \gamma + \epsilon,$$
which, by (\ref{eq:four}),
$$\leq P_n \left[\left(\cup_{k \in K_0} \left[a_k,a_{k+1}\right]\right)^{(\lambda \gamma /2)}\right] + \gamma + \epsilon,$$
which, by (\ref{eq:five}) and (\ref{eq:six}),
$$\leq P_n \left[A^{(\lambda \gamma)}\right] + \gamma + \epsilon.$$
By definition of the parametrized Prokhorov metric, we conclude that
$$\rho_\lambda(P,P_n) \leq \gamma + \epsilon.$$
Since $n$, $\lambda$, and $\epsilon$ were arbitrary, we infer that
$$\sup_{\lambda > 0} \limsup_{n} \rho_{\lambda} (P,P_n) \leq \gamma.$$
Finally, since $\gamma$ was arbitrarily taken such that (\ref{eq:seven}) holds, we conclude that the left-hand side of (\ref{eq:Ph}) is dominated by the right-hand side of (\ref{eq:Ph}).

For the reverse inequality, suppose that, for $\gamma > 0$,
\begin{equation}
\sup_{\lambda > 0} \limsup_{n \rightarrow \infty} \rho_\lambda(P,P_n) < \gamma.\label{eq:0}
\end{equation}
Now fix $h \in \mathcal{H}$ and $\epsilon > 0$. Notice that, by the monotone convergence theorem,
$\int_0^1 P\left[\{h \geq t\}^{(\lambda \gamma)}\right] dt \rightarrow \int_0^1 P[h \geq t] dt$  as  $\lambda \downarrow 0,$ whence we find $\lambda_0 > 0$ such that 
\begin{equation}
\int_0^1 P\left[\{h \geq t\}^{(\lambda_0 \gamma)}\right] dt \leq \int_0^1 P[h \geq t] dt + \epsilon.\label{eq:2}
\end{equation}
By (\ref{eq:0}), and using symmetry of $\rho_\lambda$, we find $n_0$ such that for each $n \geq n_0$ and each Borel set $A \subset \mathbb{R}$
\begin{equation}
P_n[A] \leq P\left[A^{(\lambda_0 \gamma)}\right] + \gamma.\label{eq:1}
\end{equation}
Fix $n \geq n_0$. By the layer cake representation,
$$\int h dP_n = \int_0^1 P_n[h \geq t] dt,$$
which, by (\ref{eq:1}),
$$\leq \int_0^1 P\left[\{h \geq t\}^{(\lambda_0 \gamma)}\right] dt + \gamma,$$
which, by (\ref{eq:2}),
$$\leq \int_0^1 P[h \geq t] dt + \gamma + \epsilon,$$
which, again by the layer cake representation,
$$= \int h dP + \gamma + \epsilon.$$
As before, by the arbitrariness of $n$, $h$, $\epsilon$, and $\gamma$, we conclude that
\begin{equation}
\limsup_{n \rightarrow \infty} \left(\int h d P_n - \int h dP\right) \leq \sup_{\lambda > 0} \limsup_{n \rightarrow \infty} \rho_\lambda(P,P_n).\label{eq:3}
\end{equation}
Finally, taking into account that $h \in \mathcal{H}$ if and only if $1 -h \in \mathcal{H},$ (\ref{eq:3}) learns that the right-hand side of (\ref{eq:Ph}) is dominated by the left-hand side of (\ref{eq:Ph}), which finishes the proof.
\end{proof}

\begin{thm}\label{thm:ACLTP}
Let $\xi$ be as in Section \ref{sec:Intro}. Then there exists a universal constant $C_P > 0$ such that 
\begin{equation*}
\lambda_{\rho_{\lambda}}\left(\sum_{k=1}^n \xi_{n,k} \rightarrow \xi\right) \leq C_P \textrm{\upshape{Lin}}\left(\{\xi_{n,k}\}\right)
\end{equation*}
for all $\lambda> 0$ and all STA's $\{\xi_{n,k}\}$ which satisfy Feller's condition (\ref{eq:FelCon}). Moreover, $C_P$ can be taken equal to $4$.
\end{thm}

\begin{proof}
Let $h : \mathbb{R} \rightarrow [0,1]$ be a continuously differentiable map with a bounded derivative. Then
\begin{equation}
\sup_{x_1,x_2 \in \mathbb{R}} \left|f_h^\prime(x_1) - f_h^\prime(x_2)\right| \leq 2 \|f_h^\prime\|_\infty \leq 4 \|\mathbb{E}[h(\xi)] - h\|_\infty \leq 4 ,\label{eq:ACLTP1}
\end{equation}
by (\ref{eq:lemACLTP}), and
\begin{equation}
\sup_{x_1,x_2 \in \mathbb{R}} \left|f_h^{\prime \prime}(x_1) - f_h^{\prime \prime} (x_2)\right| \leq 2 \|f_h^{\prime \prime}\|_\infty \leq 4 \|h^\prime\|_\infty\label{eq:ACLTP2},
\end{equation}
by (\ref{eq:lemACLTW2}). Furthermore, combining (\ref{eq:BigIneq}) with (\ref{eq:lemACLTW2}), (\ref{eq:ACLTP1}), and (\ref{eq:ACLTP2}), yields
\begin{eqnarray}
\lefteqn{\left|\mathbb{E}\left[h(\xi) - h\left(\sum_{k=1}^n \xi_{n,k}\right)\right]\right|}\label{eq:ACLTP3}\\
&\leq&  \|h^\prime\|_\infty \epsilon + 4 \sum_{k=1}^n \mathbb{E}\left[\xi_{n,k}^2 ; \left|\xi_{n,k}\right| \geq \epsilon\right] + 4 \|h^\prime\|_\infty \max_{k=1}^n \mathbb{E}\left[\left|\xi_{n,k}\right|\right]\nonumber
\end{eqnarray}
for all STA's $\{\xi_{n,k}\}$, all $n$, and all $\epsilon > 0$. Finally, assuming that $\{\xi_{n,k}\}$ satisfies Feller's condition, calculating the superior limits, and letting $\epsilon \downarrow 0$, we see that that (\ref{eq:Ph}) and (\ref{eq:ACLTP3}) lead to the desired result.
\end{proof}

Notice the remarkable fact that the constant $C_P$ in Theorem \ref{thm:ACLTP} does not depend on the parameter $\lambda$. This, in light of relation (\ref{LimProkhTV}),  suggests that an approximate central limit theorem in the spirit of Theorem \ref{thm:ACLTK} for the total variation distance $d_{TV}$ might be derived from Theorem \ref{thm:ACLTP}. However, the following example shows that this is not the case.

\begin{vb}
Let $\xi$ and $\{\xi_{n,k}\}$ be as in Section 1, and assume that $\{\xi_{n,k}\}$ consists of discrete random variables and satisfies Lindeberg's condition. Then 
\begin{equation*}
\textrm{\upshape{Lin}}(\{\xi_{n,k}\}) = 0
\end{equation*}
and, each $\sum_{k=1}^n \xi_{n,k}$ also being discrete, 
\begin{equation*}
d_{TV}\left(\xi,\sum_{k=1}^n \xi_{n,k}\right) = 1.
\end{equation*}
We conclude that there does not exist a constant $C > 0$ such that
\begin{equation*}
\lambda_{d_{TV}}\left(\sum_{k=1}^n \xi_{n,k} \rightarrow \xi\right) \leq C \textrm{\upshape{Lin}}(\{\xi_{n,k}\})
\end{equation*}
for all STA's $\{\xi_{n,k}\}$ satisfying Feller's condition.
\end{vb}

We summarize the information obtained in Theorem \ref{thm:ACLTK} and Remark \ref{rem:ACLTK}, Theorem \ref{thm:ACLTW}, and Theorem \ref{thm:ACLTP}, in the following result. We put
\begin{equation*}
\lambda_P\left(\sum_{k=1}^n \xi_{n,k} \rightarrow \xi\right) = \sup_{\lambda \in \mathbb{R}^+_0} \lambda_{\rho_\lambda}\left(\sum_{k=1}^n \xi_{n,k} \rightarrow \xi\right).
\end{equation*}

\begin{thm}\label{thm:ACLT}
Let $\xi$ be as in Section 1. Then, for each $\delta \in \{K,W,P\}$, there exists a universal constant $C_\delta > 0$ such that
\begin{equation*}
\lambda_\delta\left(\sum_{k=1}^n \xi_{n,k} \rightarrow \xi\right) \leq C_\delta \textrm{\upshape{Lin}}\left(\left\{\xi_{n,k}\right\}\right)
\end{equation*}
for all STA's $\{\xi_{n,k}\}$ satisfying Feller's condition (\ref{eq:FelCon}). Moreover, $C_K$ can be taken equal to $1$, $C_W$ equal to $8$, and $C_P$ equal to $4$.
\end{thm}

\section*{Appendix A: Proof of Theorem \ref{thm:BigIneq}}\label{appendix:ProofIneq}

We follow \cite{BLV13}, Section 2. We keep a continuously differentiable $h : \mathbb{R} \rightarrow [0,1]$, with bounded derivative, fixed, and let $f_h$ be its Stein transform defined by (\ref{eq:SteinTransform}). Also, we put $$\sigma_{n,k}^2 = \mathbb{E}[\xi_{n,k}^2].$$

The following lemma is easily verified. It can be found in e.g. \cite{BC05} (p.10-11).

\begin{lemA}\label{SteinBasics}
$f_h$ is twice continuously differentiable, has bounded first and second derivatives, and
\begin{eqnarray}
\mathbb{E}\left[h(\xi)\right] - h(x)= x f_h(x) - f_h^\prime (x).\label{SteinIdentity}
\end{eqnarray}
\end{lemA}

The following lemma can be found in \cite{BLV13} (Lemma 2.4). We give the proof for completeness.

\begin{lemA}\label{BHDELem}
Put
\begin{eqnarray*}
\delta_{n,k} = f_h\left(\sum_{i \neq k} \xi_{n,i} + \xi_{n,k}\right) - f_{h}\left(\sum_{i \neq k} \xi_{n,i}\right) - \xi_{n,k} f^\prime_h\left(\sum_{i \neq k} \xi_{n,i}\right)\label{defdelta}
\end{eqnarray*}
and 
\begin{eqnarray*}
\epsilon_{n,k} = f^\prime_h\left(\sum_{i \neq k} \xi_{n,i} + \xi_{n,k}\right) - f^\prime_{h}\left(\sum_{i \neq k} \xi_{n,i}\right) - \xi_{n,k} f^{\prime \prime}_h\left(\sum_{i \neq k}\xi_{n,i}\right).\label{defepsilon}
\end{eqnarray*}
Then
\begin{eqnarray}
\lefteqn{\mathbb{E}\left[\left(\sum_{k=1}^n \xi_{n,k}\right) f_h\left(\sum_{k=1}^n \xi_{n,k}\right) - f_h^\prime\left(\sum_{k=1}^n \xi_{n,k}\right)\right]}\nonumber\\
 &=& \sum_{k=1}^n \mathbb{E}\left[\xi_{n,k}\delta_{n,k}\right] - \sum_{k=1}^n \sigma_{n,k}^2 \mathbb{E}\left[\epsilon_{n,k}\right].\label{BHallIneq}
\end{eqnarray}
\end{lemA}

\begin{proof}
Recalling that $\xi_{n,k}$ and $\sum_{i \neq k} \xi_{n,i}$ are independent, $\mathbb{E}\left[\xi_{n,k}\right] = 0$, and $\sum_{k=1}^n \sigma_{n,k}^2 = 1$, we get
\begin{eqnarray*}
\lefteqn{\sum_{k=1}^n \mathbb{E}\left[\xi_{n,k} \delta_{n,k}\right] - \sum_{k=1}^n \sigma_{n,k}^2 \mathbb{E}\left[\epsilon_{n,k}\right]}\\
&=& \sum_{k=1}^n \mathbb{E}\left[\xi_{n,k} f_{h}\left(\sum_{k = 1}^n \xi_{n,k}\right)\right] - \mathbb{E}\left[\xi_{n,k} f_{h}\left(\sum_{i \neq k} \xi_{n,i}\right)\right]\\
&& - \sum_{k=1}^n \mathbb{E}\left[\xi_{n,k}^2 f^\prime_h\left(\sum_{i \neq k} \xi_{n,i}\right)\right] - \sum_{k=1}^n \sigma_{n,k}^2 \mathbb{E}\left[f_{h}^{\prime}\left(\sum_{k=1}^n \xi_{n,k}\right)\right]\\
&& + \sum_{k=1}^n \mathbb{E}\left[\xi_{n,k}^2\right] \mathbb{E}\left[f^\prime_{h}\left(\sum_{i \neq k} \xi_{n,i}\right)\right] + \sum_{k=1}^n \sigma_{n,k}^2 \mathbb{E}\left[\xi_{n,k} f^{\prime \prime}_h\left(\sum_{i \neq k}\xi_{n,i}\right)\right] . 
\end{eqnarray*}
The last expression further reduces to
\begin{eqnarray*}
\lefteqn{\mathbb{E}\left[\left(\sum_{k=1}^n\xi_{n,k}\right) f_{h}\left(\sum_{k = 1}^n \xi_{n,k}\right)\right] - \mathbb{E}\left[\xi_{n,k}\right] \mathbb{E}\left[f_{h}\left(\sum_{i \neq k} \xi_{n,i}\right)\right]}\\
&& - \sum_{k=1}^n \mathbb{E}\left[\xi_{n,k}^2 f^\prime_h\left(\sum_{i \neq k} \xi_{n,i}\right)\right] -  \mathbb{E}\left[f_{h}^{\prime}\left(\sum_{k=1}^n \xi_{n,k}\right)\right]\\
&& + \sum_{k=1}^n \mathbb{E}\left[\xi_{n,k}^2 f^\prime_{h}\left(\sum_{i \neq k} \xi_{n,i}\right)\right] + \sum_{k=1}^n \sigma_{n,k}^2 \mathbb{E}\left[\xi_{n,k}\right] \mathbb{E}\left[f^{\prime \prime}_h\left(\sum_{i \neq k}\xi_{n,i}\right)\right] , 
\end{eqnarray*}
which is easily seen to equal
\begin{displaymath}
\mathbb{E}\left[\left(\sum_{k=1}^n \xi_{n,k}\right) f_h\left(\sum_{k=1}^n \xi_{n,k}\right) - f_h^\prime\left(\sum_{k=1}^n \xi_{n,k}\right)\right].
\end{displaymath}
This finishes the proof.
\end{proof}

The following lemma is an application of Taylor's theorem.

\begin{lemA}
For any $a, x \in \mathbb{R}$,
\begin{eqnarray}
\lefteqn{\left|f_h(a + x) - f_h(a) - f_h^\prime(a) x \right|}\nonumber\\ 
&&\leq \min \left\{\left(\sup_{x_1,x_2 \in \mathbb{R}}\left|f_h^\prime(x_1) - f^{\prime}(x_2)\right|\right) \left|x\right|,\frac{1}{2} \left\|f_h^{\prime \prime}\right\|_\infty x^2\right\}.\label{Taylorf}
\end{eqnarray}
\end{lemA}

We are now in a position to present a proof of Theorem \ref{thm:BigIneq}.\\

{\em Proof of Theorem \ref{thm:BigIneq}.}
For $n$ and $\epsilon > 0$, we have, by (\ref{SteinIdentity}), (\ref{BHallIneq}), and (\ref{Taylorf}),
\begin{eqnarray*}
\lefteqn{\left|\mathbb{E}\left[h\left(\xi\right) - h\left(\sum_{k=1}^{n}\xi_{n,k}\right) \right]\right|}\\
 &=& \left|\mathbb{E}\left[\left(\sum_{k=1}^n \xi_{n,k}\right) f_h\left(\sum_{k=1}^n \xi_{n,k}\right) - f_h^\prime\left(\sum_{k=1}^n \xi_{n,k}\right)\right]\right|\\
&\leq&  \sum_{k=1}^n \mathbb{E}\left[\left|\xi_{n,k}\delta_{n,k}\right|\right] + \sum_{k=1}^n \sigma_{n,k}^2 \mathbb{E}\left[\left|\epsilon_{n,k}\right|\right]\\
&\leq& \frac{1}{2}\left\|f_h^{\prime \prime}\right\|_\infty \sum_{k=1}^n \mathbb{E}\left[\left|\xi_{n,k}\right|^3 ; \left|\xi_{n,k}\right|<\epsilon\right]\\
&& + \left(\sup_{x_1,x_2 \in \mathbb{R}} \left|f_h^\prime(x_1) - f_{h}^\prime(x_2)\right|\right) \sum_{k=1}^n \mathbb{E}\left[\left|\xi_{n,k}\right|^2;\left|\xi_{n,k}\right|\geq \epsilon\right]\\
&& + \left(\sup_{x_1, x_2 \in \mathbb{R}} \left|f_h^{\prime \prime}(x_1) - f_h^{\prime \prime} (x_2)\right| \right)\sum_{k=1}^n\sigma_{n,k}^2 \mathbb{E}\left[\left|\xi_{n,k}\right|\right],
\end{eqnarray*}
which proves the desired result since $\sum_{k=1}^n \sigma_{n,k}^2 = 1$.

\end{document}